\documentclass[12pt]{article} 
\usepackage{amscd,amsmath,amssymb,amsfonts,braket,mathrsfs,latexsym,enumerate,amsthm}
\usepackage[all]{xy}
\usepackage{graphics}
\usepackage[dvips]{graphicx}
\usepackage{eepic}
\usepackage{epic}
\newtheorem{thm}{Theorem}[section]

\newtheorem{lmm}[thm]{Lemma}

\newtheorem{prp}[thm]{Proposition}

\theoremstyle{definition}
\newtheorem{dfn}[thm]{Definition}
\newtheorem{conj}[thm]{Conjecture}
\theoremstyle{remark}
\newtheorem*{rem}{Remark}

\numberwithin{thm}{section}
\numberwithin{equation}{section}

\newcommand{\C}{\mathbb{C}}
\newcommand{\F}{\mathbb{F}}

\newcommand{\Q}{\mathbb{Q}}

\newcommand{\Kb}{\overline{K}}

\newcommand{\kb}{\overline{k}}
\newcommand{\ksep}{k^{\text{sep}}}
\newcommand{\R}{\mathbb{R}}
\newcommand{\Z}{\mathbb{Z}}
\newcommand{\pP}{\mathbb{P}}

\newcommand{\rhob}{\overline{\rho}}

\newcommand{\mfq}{\mathfrak{q}}

\newcommand{\cO}{\mathcal{O}}

\newcommand{\cQM}{\mathcal{QM}}

\newcommand{\Gal}{\mathrm{Gal}}
\newcommand{\G}{\mathrm{G}}
\newcommand{\M}{\mathrm{M}}
\newcommand{\GL}{\mathrm{GL}}

\newcommand{\End}{\mathrm{End}}
\newcommand{\Aut}{\mathrm{Aut}}

\newcommand{\Norm}{\mathrm{Norm}}

\newcommand{\cf}{cf.\ }
\newcommand{\inj}{\hookrightarrow}

\newcommand{\resp}{resp.\ }

\newcommand{\disc}{\mathrm{disc}}
\newcommand{\id}{\mathrm{id}}

\newcommand{\cM}{\mathcal{M}}
\newcommand{\cN}{\mathcal{N}}
\newcommand{\cFR}{\mathcal{FR}}

\newcommand{\cS}{\mathcal{S}}
\newcommand{\cT}{\mathcal{T}}

\newcommand{\Ram}{\mathbf{Ram}}
\newcommand{\N}{\mathrm{N}}

\newcommand{\Ktil}{\widetilde{K}}

\newcommand{\sA}{\mathscr{A}}

\newcommand{\pmu}{\pmb{\mu}}

\newcommand{\ba}{\boldsymbol{a}}
\newcommand{\bc}{\boldsymbol{c}}
\newcommand{\bb}{\boldsymbol{b}}
\newcommand{\bo}{\boldsymbol{0}}

\title{On the Rasmussen-Tamagawa conjecture for QM-abelian surfaces}
\date{}
\author{Keisuke \textsc{Arai}\footnote{School of Engineering, Tokyo Denki University, Tokyo
120-8551, Japan.\newline e-mail: \texttt{araik@mail.dendai.ac.jp}}
}
\begin{document}
%

\maketitle


\begin{abstract}      
In a previous article, we showed the Rasmussen-Tamagawa conjecture for QM-abelian
surfaces over imaginary quadratic fields.
In this article, we generalize the previous work to QM-abelian surfaces
over number fields of higher degree.
We also give several explicit examples.
\end{abstract}


\section{Introduction}
\label{secIntro}

For a number field $K$ and
a prime number $p$,
let $\Kb$ denote an algebraic closure of $K$,
and let $\Ktil_p$ denote the maximal pro-$p$
extension of $K(\pmu_p)$ in $\Kb$ which is unramified away from $p$,
where $\pmu_p$ is the group of $p$-th roots of unity in $\Kb$.
For a number field $K$, an integer $g\geq 0$ and a prime number $p$, let
$\sA(K,g,p)$ denote the set of $K$-isomorphism classes of abelian varieties
$A$ over $K$, of dimension $g$, which satisfy 
\begin{equation}
\label{AKgp}
K(A[p^{\infty}])\subseteq\Ktil_p,
\end{equation}
where $K(A[p^{\infty}])$ is the subfield of $\Kb$ generated over $K$ by the $p$-power torsion of $A$.
 %
%
%
It follows from \cite[Theorem 1, p.493]{SeTa} that
an abelian variety $A$ over $K$
has good reduction at any prime of $K$ not dividing $p$
if its class belongs to $\sA(K,g,p)$,
because the extension $K(A[p^{\infty}])/K(\pmu_p)$
is unramified away from $p$.
%
So we can conclude that $\sA(K,g,p)$ is a finite set
(\cite[1. Theorem, p.309]{Z}, \cf\cite[Satz 6, p.363]{Fa}).
For fixed $K$ and $g$, define the set
$$\sA(K,g):=\{([A],p)\mid \text{$p$ : prime number, } [A]\in\sA(K,g,p)\}.$$
We have the following conjecture concerning finiteness for abelian varieties,
which is called the Rasmussen-Tamagawa conjecture (\cite[p.2391]{Oz1}):

\begin{conj}[{\cite[Conjecture 1, p.1224]{RaTa}}]
\label{RTconj}

Let $K$ be a number field, and let $g\geq 0$ be an integer.
Then the set $\sA(K,g)$ is finite.


\end{conj}

For elliptic curves, we have the following result related to
Conjecture \ref{RTconj} 
(owing to \cite[Theorem 7.1, p.153]{Ma3} and \cite[Theorem B, p.330]{Mo4}):

\begin{thm}[{\cite[Theorem 2, p.1224 and Theorem 4, p.1227]{RaTa}}]
\label{A(Q,1)}

Let $K$ be $\Q$ or a quadratic field which is not an imaginary quadratic field
of class number one.
Then the set $\sA(K,1)$ is finite.

\end{thm}







We are interested in higher dimensional cases,
in particular,
in the case of QM-abelian surfaces,
which are analogous to elliptic curves.
Let $B$ be an indefinite quaternion division algebra over $\Q$. 
Let
$$d=\disc(B)$$
be the discriminant of $B$. 
Then $d>1$ and $d$ is the product of an even number of distinct prime numbers.
Choose and fix a maximal order $\cO$ of $B$. 
If a prime number $p$ does not divide $d$, fix an isomorphism 
$$\cO\otimes_{\Z}\Z_p\cong \M_2(\Z_p)$$
of $\Z_p$-algebras.
Now we recall the definition of QM-abelian surfaces.

\begin{dfn}[\cf {\cite[p.591]{Bu}}]
\label{defqm}

Let $S$ be a scheme over $\Q$. 
A QM-abelian surface by $\cO$
over $S$ is a pair $(A,i)$ where $A$ is an 
abelian surface over $S$ (i.e. $A$ is an abelian scheme over $S$ 
of relative dimension $2$), and 
$i:\cO\inj\End_S(A)$ 
is an injective ring homomorphism (sending $1$ to $\id$).
Here $\End_S(A)$ is the ring of endomorphisms of $A$ defined over $S$.
We assume that $A$ has a left $\cO$-action.
We will sometimes omit ``by $\cO$" and simply write
``a QM-abelian surface"
if there is no fear of confusion.

\end{dfn}

For a number field $K$ and a prime number $p$,
let $\sA(K,2,p)_B$ be the set of $K$-isomorphism classes of
abelian varieties $A$ over $K$ in $\sA(K,2,p)$
such that there is an injective ring homomorphism
$\cO\inj\End_K(A)$
(sending $1$ to $\id$).
Let us also define the set
$$\sA(K,2)_B:=\{([A],p)\mid \text{$p$ : prime number, } [A]\in\sA(K,2,p)_B\}.$$
Let $h_K$ denote the class number of $K$.
%
Conjecture \ref{RTconj} for QM-abelian surfaces
has been partly confirmed.

\begin{thm}[{\cite[Theorem 9.3]{AM}}, \cf\cite{AM2}]
\label{RTB}

Let $K$ be an imaginary quadratic field with $h_K\geq 2$.
Then the set $\sA(K,2)_B$ is finite.

\end{thm}


The main result of this article is the following theorem,
which is a generalization of Theorem \ref{RTB}
to number fields of higher degree.

\begin{thm}
\label{main}

Let $K$ be a finite Galois extension of $\Q$ which does not contain
the Hilbert class field of any imaginary quadratic field.
Assume that there is a prime number $q$ which
splits completely in $K$ and satisfies
$B\otimes_{\Q}\Q(\sqrt{-q})\not\cong\M_2(\Q(\sqrt{-q}))$.
Then the set $\sA(K,2)_B$ is finite.

\end{thm}

In the next section, we prove Theorem \ref{main}.
In \S \ref{secEx}, we give examples of the main result after recalling needed facts in \S \ref{secMB}.


\begin{rem}
\label{Ozeki}

(1)
The condition (\ref{AKgp}) is equivalent to the following assertion
(see \cite[Lemma 3, p.1225]{RaTa} or \cite[Definition 4.1, p.2390]{Oz1}):

\noindent
The abelian variety $A$ has good reduction outside $p$,
and
the group
$A[p](\Kb)$
consisting of $p$-torsion points of $A$
has a filtration of $\G_K$-modules
$\{ 0\}=V_0\subseteq V_1\subseteq\cdots\subseteq V_{2g-1}\subseteq V_{2g}=A[p](\Kb)$
such that $V_i$ has dimension $i$ for each $1\leq i\leq 2g$,
where $\G_K$ is the absolute Galois group of $K$.
Furthermore, for each $1\leq i\leq 2g$, there is an integer
$a_i\in\Z$ such that $\G_K$ acts on $V_i/V_{i-1}$
by $gv=\theta_p(g)^{a_i}v$,
where $g\in\G_K$, $v\in V_i/V_{i-1}$, and $\theta_p$ is the mod $p$ cyclotomic character.

(2)
Conjecture \ref{RTconj} is equivalent to the following assertion:

\noindent
There exists a constant $C_{\text{RT}}(K,g)>0$ depending on $K$ and $g$ such that
we have $\sA(K,g,p)=\emptyset$
for any prime number $p>C_{\text{RT}}(K,g)$.

(3)
The set $\sA(K,2,p)_B$ (\resp $\sA(K,2)_B$)
is a subset of $\sA(K,2,p)$ (\resp $\sA(K,2)$).
If one of the following two conditions
is satisfied, we know that the sets $\sA(K,2,p)_B$, $\sA(K,2)_B$ are empty
for a trivial reason: there are no QM-abelian surfaces by $\cO$ over $K$
(\cite[Theorem 0, p.136]{Sh}, \cite[Theorem (1.1), p.93]{J}).

\noindent
(i) $K$ has a real place.

\noindent
(ii) $B\otimes_{\Q}K\not\cong\M_2(K)$.

(4)
Let $\cQM$ be the set of isomorphism classes of indefinite quaternion division
algebras over $\Q$.
Define the set
$$\sA(K,2)_{\cQM}:=\bigcup_{B\in\cQM}\sA(K,2)_B$$  
which is a subset of $\sA(K,2)$.
We then have the following corollary to Theorem \ref{RTB} (see \cite[Corollary 9.5]{AM}):

\noindent
Let $K$ be an imaginary quadratic field with $h_K\geq 2$.
Then the set $\sA(K,2)_{\cQM}$ is finite.

(5)
Conjecture \ref{RTconj} is solved for any $K$ and $g$
when restricted to semi-stable abelian varieties (\cite[Corollary 4.5, p.2392]{Oz1})
or abelian varieties
with abelian Galois representations (\cite[Theorem 1.2]{Oz2}).
See also \cite[\S 6]{A2} for a summary.

\end{rem}

\vspace{5mm}
\noindent
{\bf Notation}

\vspace{1mm}

For a field $k$,
let $\kb$ denote an algebraic closure of $k$,
let $\ksep$ 
denote the separable closure
of $k$ inside $\kb$,
and let $\G_k=\Gal(\ksep/k)$.
%

For an integer $n\geq 1$ and a commutative group (or a commutative group scheme) $G$,
let $G[n]$ denote the kernel of multiplication by $n$ in $G$.
%

%

For a prime number $p$ and an abelian variety $A$ over a field $k$,
let
$\displaystyle T_pA:=\lim_{\longleftarrow}A[p^n](\kb)$
be the $p$-adic Tate module of $A$,
where the inverse limit is taken with respect to
multiplication by $p$ : $A[p^{n+1}](\kb)\longrightarrow A[p^n](\kb)$.

For a number field $K$, 
let $\cO_K$ denote the ring of integers of $K$,
let $K_v$ denote the completion of $K$ at $v$
where $v$ is a place (or a prime) of $K$,
and let $\Ram (K)$ denote the set of prime numbers which are ramified in $K$.


\section{Galois representations}
\label{secGalrep}

A QM-abelian surface 
has a Galois representation which looks like 
that of an elliptic curve
as explained below (\cf\cite{Oh}). 
Let $k$ be a field of characteristic $0$, and
let $(A,i)$ be a QM-abelian surface by $\cO$ over $k$,
where $\cO$ is a fixed maximal order of $B$ which is
an indefinite quaternion division algebra over $\Q$.
We consider the Galois representations associated to $(A,i)$.
Take a prime number $p$ not dividing $d=\disc(B)$.
We then have isomorphisms of $\Z_p$-modules:
$$\Z_p^4\cong T_pA\cong\cO\otimes_{\Z}\Z_p\cong\M_2(\Z_p).$$
The middle is also an isomorphism of left $\cO$-modules;
the last is also an isomorphism of $\Z_p$-algebras (which was fixed in \S\ref{secIntro}).
We sometimes identify these $\Z_p$-modules.
Take a $\Z_p$-basis
$$e_1=\left(\begin{matrix} 1\ &0 \\ 0\ &0\end{matrix}\right),\ 
e_2=\left(\begin{matrix} 0\ &1 \\ 0\ &0\end{matrix}\right),\ 
e_3=\left(\begin{matrix} 0\ &0 \\ 1\ &0\end{matrix}\right),\ 
e_4=\left(\begin{matrix} 0\ &0 \\ 0\ &1\end{matrix}\right)$$
of $\M_2(\Z_p)$.
Then the image of the natural map
$$\M_2(\Z_p)\cong\cO\otimes_{\Z}\Z_p\hookrightarrow\End(T_pA)\cong\M_4(\Z_p)$$
lies in
$\left\{\left(\begin{matrix} aI_2\ &bI_2 \\ cI_2\ &dI_2\end{matrix}\right)
\Bigg|\, a,b,c,d\in\Z_p\right\}$,
where 
$I_2=\left(\begin{matrix} 1\ &0 \\ 0\ &1\end{matrix}\right)$.
%
The $\G_k$-action on $T_pA$ induces a representation
$$\rho_{A/k,p}:\G_k\longrightarrow\Aut_{\cO\otimes_{\Z}\Z_p}(T_pA)\subseteq\Aut(T_pA)
\cong\GL_4(\Z_p),$$
where
$\Aut_{\cO\otimes_{\Z}\Z_p}(T_pA)$
is the group of automorphisms of $T_pA$
commuting with the action of $\cO\otimes_{\Z}\Z_p$.
The above observation implies
$$\Aut_{\cO\otimes_{\Z}\Z_p}(T_pA)=
\left\{\left(\begin{matrix} 
X\ &0 \\ 0\ &X
\end{matrix}\right)
\Bigg|\, X\in\GL_2(\Z_p)\right\}
\subseteq\GL_4(\Z_p).$$
Then the representation $\rho_{A/k,p}$ factors as
$$\rho_{A/k,p}:\G_k\longrightarrow
\left\{\left(\begin{matrix} X\ &0 \\ 0\ &X\end{matrix}\right)
\Bigg|\, X\in\GL_2(\Z_p)\right\}
\subseteq\GL_4(\Z_p).$$
Let
$$\rho_{(A,i)/k,p}:\G_k\longrightarrow \GL_2(\Z_p)$$
denote the Galois representation determined by
``$X$",
so that we have
$\rho_{(A,i)/k,p}(\sigma)=X(\sigma)$
if $\rho_{A/k,p}(\sigma)=\left(\begin{matrix} X(\sigma)\ &0 \\ 0\ &X(\sigma)\end{matrix}\right)$
for $\sigma\in\G_k$.
Let
$$\overline{\rho}_{A/k,p}:\G_k\longrightarrow \GL_4(\F_p)\ \ \ \ \ (\text{\resp}\ \ 
\overline{\rho}_{(A,i)/k,p}:\G_k\longrightarrow \GL_2(\F_p))$$
denote the reduction of
$\rho_{A/k,p}$ (\resp $\rho_{(A,i)/k,p}$) modulo $p$. 
Note that this construction of $\overline{\rho}_{(A,i)/k,p}$ is slightly different
from that in \cite[\S 2]{AM}, but the resulting representations are the same.
%

We have the following criterion for Conjecture \ref{RTconj}
for QM-abelian surfaces.

\begin{lmm}
\label{irredRT}

Assume that there is a constant $C(B,K)$ depending on $B$ and a number field $K$ such that
$\rhob_{(A,i)/K,p}$ is irreducible for any prime number $p>C(B,K)$ and
any QM-abelian surface $(A,i)$ by $\cO$ over $K$.
Then the set $\sA(K,2)_B$ is finite.

\end{lmm}

\begin{proof}

Take an element $([A],p)\in\sA(K,2)_B$.
Since $[A]\in\sA(K,2,p)$,
we know that $\rhob_{A/K,p}$ is conjugate to
$\left(\begin{matrix}
*\ &*\ &*\ &* \\ 0\ &*\ &*\ &* \\ 0\ &0\ &*\ &* \\ 0\ &0\ &0\ &* \\
\end{matrix}\right)$.
By the definition of $\sA(K,2)_B$, there is an embedding
$i:\cO\hookrightarrow\End_K(A)$.
Then $(A,i)$ is a QM-abelian surface by $\cO$ over $K$.
We have seen that there is a map
$X:\G_K\longrightarrow\GL_2(\F_p)$
such that
$\rhob_{A/K,p}(\sigma)=
\left(\begin{matrix} X(\sigma)\ &0 \\ 0\ &X(\sigma) \end{matrix}\right)$
for any $\sigma\in\G_K$.
Then there is a matrix
$M=\left(\begin{matrix} M_1\ &M_2 \\ M_3\ &M_4 \end{matrix}\right)\in\GL_4(\F_p)$
(where $M_1,M_2,M_3,M_4$ are $2\times 2$ matrices)
such that
$M^{-1}\left(\begin{matrix} X(\sigma)\ &0 \\ 0\ &X(\sigma) \end{matrix}\right)M
\in\Set{
\left(\begin{matrix}
*\ &*\ &*\ &* \\ 0\ &*\ &*\ &* \\ 0\ &0\ &*\ &* \\ 0\ &0\ &0\ &* \\
\end{matrix}\right)
}$
for any $\sigma\in\G_K$.

We claim the following.
\begin{description}
\item
(C): There is a matrix $H\in\GL_2(\F_p)$ such that
$H^{-1}X(\sigma)H\in
\Set{\left(\begin{matrix} *\ &* \\ 0\ &* \end{matrix}\right)}$
for any $\sigma\in\G_K$.
\end{description}
Let $M_1=(\ba_1\ \ba_2)$, $M_3=(\bc_1\ \bc_2)$ and
$M^{-1}\left(\begin{matrix} X(\sigma)\ &0 \\ 0\ &X(\sigma) \end{matrix}\right)M
=\left(\begin{matrix}
s(\sigma)\ &t(\sigma)\ &*\ &* \\ 0\ &u(\sigma)\ &*\ &* \\ 0\ &0\ &*\ &* \\ 0\ &0\ &0\ &* \\
\end{matrix}\right)$.
Then
$X(\sigma)\ba_1=s(\sigma)\ba_1$,
$X(\sigma)\ba_2=t(\sigma)\ba_1+u(\sigma)\ba_2$,
$X(\sigma)\bc_1=s(\sigma)\bc_1$, and
$X(\sigma)\bc_2=t(\sigma)\bc_1+u(\sigma)\bc_2$
for any $\sigma\in\G_K$.
If $\ba_1\ne \bo$, take a vector $\bb\in\F_p^2$ not contained in
the linear subspace $\F_p\ba_1$ and put $H=(\ba_1\ \bb)$.
Then (C) holds.
If $\ba_1=\bo$ and $\ba_2\ne \bo$, then
$X(\sigma)\ba_2=u(\sigma)\ba_2$,
and so (C) holds.
If $\ba_1=\ba_2=\bo$, then $\bc_1\ne\bo$ or $\bc_2\ne\bo$
because the matrix
$M=\left(\begin{matrix} M_1\ &M_2 \\ M_3\ &M_4 \end{matrix}\right)$
is invertible.
Then (C) follows.

In this case $\rhob_{(A,i)/K,p}$ is reducible, and so $p\leq C(B,K)$.
Therefore $\sharp\sA(K,2)_B<\infty$.

\end{proof}

Theorem \ref{main} is a consequence of the following theorem
(Theorem \ref{galrepirred}) together with Lemma \ref{irredRT}.
Before stating this theorem, we need some preparation.
For a number field $K$,
let $\cM$ be the set of prime numbers $q$ such that $q$ splits
completely in $K$
and $q$ does not divide $6h_K$.
Let $\cN$ be the set of primes $\mfq$ of $K$ such that
$\mfq$ divides some prime number $q\in \cM$.
Take a finite subset $\emptyset\ne \cS\subseteq \cN$ such that
$\cS$ generates the ideal class group of $K$. 
For each prime $\mfq\in \cS$, fix an element $\alpha_{\mfq}\in\cO_K\setminus\{0\}$
satisfying $\mfq^{h_K}=\alpha_{\mfq}\cO_K$.
For a prime number $q$, put
$$\cFR(q):=\Set{\beta\in\C|
\beta^2+a\beta+q=0 \text{ for some integer $a\in\Z$ with $|a|\leq 2\sqrt{q}$}}.$$
For $\mfq\in\cS$,
put $\N(\mfq)=\sharp(\cO_K/\mfq)$.
Then $\N(\mfq)$ is a prime number.
For a finite Galois extension $K$ of $\Q$, define the sets
(\cf \cite{A3}, \cite{AM})

\noindent
$\cM'_1(K):=$
$$\Set{(\mfq,\varepsilon'_0,\beta_{\mfq})|
\mfq\in \cS,\ \varepsilon'_0=\sum_{\sigma\in\Gal(K/\Q)}a'_{\sigma}\sigma
\text{ with $a'_{\sigma}\in\{0,4,6,8,12 \}$},\ 
\beta_{\mfq}\in\cFR(\N(\mfq))}$$

\noindent
(where $\varepsilon'_0$ is an element of the group ring $\Z[\Gal(K/\Q)]$),

\noindent
$\cM'_2(K):=\Set{\Norm_{K(\beta_{\mfq})/\Q}(\alpha_{\mfq}^{\varepsilon'_0}-\beta_{\mfq}^{12h_K})\in\Z|
(\mfq,\varepsilon'_0,\beta_{\mfq})\in\cM'_1(K)}\setminus\{0\}$,

\noindent
$\cN'_0(K):=\Set{\text{$l$ : prime number}|\text{$l$ divides some integer $m\in\cM'_2(K)$}}$,

\noindent
$\cT(K):=\Set{\text{$l'$ : prime number}|\text{$l'$ is divisible
by some prime $\mfq'\in \cS$}}
\cup\{2,3\}$,

\noindent
and

\noindent
$\cN'_1(K):=\cN'_0(K)\cup\cT(K)\cup\Ram(K)$.

\noindent
Note that all the sets, $\cFR(q), \cM'_1(K)$, $\cM'_2(K)$, $\cN'_0(K)$, $\cT(K)$, and
$\cN'_1(K)$, are finite.

\begin{thm}[{\cite[Theorem 6.5]{A3}}]
\label{galrepirred}

Let $K$ be a finite Galois extension of $\Q$ which does not contain
the Hilbert class field of any imaginary quadratic field.
Assume that there is a prime number $q$ which splits completely in $K$ and satisfies
$B\otimes_{\Q}\Q(\sqrt{-q})\not\cong\M_2(\Q(\sqrt{-q}))$.
Let $p>4q$ be a prime number which also satisfies
$p\nmid d$ and $p\not\in \cN'_1(k)$.
Then the representation
$$\rhob_{(A,i)/K,p}:\G_K\longrightarrow\GL_2(\F_p)$$
is irreducible for any QM-abelian surface $(A,i)$ by $\cO$ over $K$.

\end{thm}

\section{Points on Shimura curves}
\label{secMB}

Let $M^B$ be the coarse moduli scheme over $\Q$, parameterizing isomorphism classes
of QM-abelian surfaces by $\cO$.
Then $M^B$ is a proper smooth curve over $\Q$, called a Shimura curve.
The notation $M^B$ is permissible,
although we should write $M^{\cO}$ instead of $M^B$ because,
even if we replace $\cO$ by another maximal order $\cO'$,
we have a natural isomorphism
$M^{\cO}\cong M^{\cO'}$
since $\cO$ and $\cO'$ are conjugate in $B$.
We discuss points on $M^B$, and the consequences of this section
will be used to provide examples of Theorem \ref{main} (see Proposition \ref{exRT} in \S \ref{secEx}).
For real points on $M^B$, we know the following.

\begin{thm}[{\cite[Theorem 0, p.136]{Sh}}]
\label{M^B(R)}

We have $M^B(\R)=\emptyset$.

\end{thm}

The genus of the Shimura curve $M^B$ is $0$ if and only if
$d\in\{6,10,22\}$ (\cite[Lemma 3.1, p.168]{A1}).
The defining equations of such curves are 

\begin{equation}
\label{eqMB}
\begin{cases}
d=6\ :\ x^2+y^2+3=0,\\
d=10\ :\ x^2+y^2+2=0,\\
d=22\ :\ x^2+y^2+11=0
\end{cases}
\end{equation}

\noindent
(see \cite[Theorem 1-1, p.279]{Kur}).
In these cases, for a field $k$ of characteristic $0$, the condition $M^B(k)\ne\emptyset$
implies
that the base change $M^B\otimes_{\Q}k$ is isomorphic to
the projective line $\pP^1_k$,
and so
$\sharp M^B(k)=\infty$.

\begin{thm}[{\cite[Theorem (1.1), p.93]{J}}]
\label{fieldofdefk}

Let $k$ be a field of characteristic $0$.
A point of $M^B(k)$ can be represented by a QM-abelian surface by $\cO$
over $k$ if and only if $B\otimes_{\Q}k\cong\M_2(k)$.

\end{thm}

\begin{rem}
\label{reminfty(A,i)}

For a field $k$ of characteristic $0$, note that if
$\sharp M^B(k)=\infty$ and $B\otimes_{\Q}k\cong\M_2(k)$,
then there are infinitely many $\kb$-isomorphism classes
of QM-abelian surfaces $(A,i)$ by $\cO$ over $k$.

\end{rem}


Next we quote a recent result
concerning
algebraic points on Shimura curves of $\Gamma_0(p)$-type,
which is related to Theorem \ref{galrepirred}
(but there is no implication from or to that theorem).
For a prime number $p$ not dividing $d$,
let $M_0^B(p)$
be the coarse moduli scheme over $\Q$ parameterizing isomorphism classes
of triples $(A,i,V)$ where $(A,i)$ is a QM-abelian surface by $\cO$
and $V$ is a left $\cO$-submodule of $A[p]$ with $\F_p$-dimension $2$.
Then $M_0^B(p)$ is a proper smooth curve over $\Q$, which we call a
Shimura curve of $\Gamma_0(p)$-type.
We have a natural map $M_0^B(p)\longrightarrow M^B$
over $\Q$ defined by $(A,i,V)\longmapsto (A,i)$.
So, Theorem \ref{M^B(R)}  implies $M_0^B(p)(\R)=\emptyset$ for any $p$.
%
%
%
For a finite Galois extension $K$ of $\Q$, define the finite sets

\noindent
$\cM_1(K):=$
$$\Set{(\mfq,\varepsilon_0,\beta_{\mfq})|
\mfq\in \cS,\ \varepsilon_0=\sum_{\sigma\in\Gal(K/\Q)}a_{\sigma}\sigma
\text{ with $a_{\sigma}\in\{0,8,12,16,24 \}$},\ 
\beta_{\mfq}\in\cFR(\N(\mfq))}$$

\noindent
(where $\varepsilon_0$ is an element of the group ring $\Z[\Gal(K/\Q)]$),

\noindent
$\cM_2(K):=\Set{\Norm_{K(\beta_{\mfq})/\Q}(\alpha_{\mfq}^{\varepsilon_0}-\beta_{\mfq}^{24h_K})\in\Z|
(\mfq,\varepsilon_0,\beta_{\mfq})\in\cM_1(K)}\setminus\{0\}$,

\noindent
$\cN_0(K):=\Set{\text{$l$ : prime number}|\text{$l$ divides some integer $m\in\cM_2(K)$}}$,


\noindent
and

\noindent
$\cN_1(K):=\cN_0(K)\cup\cT(K)\cup\Ram(K)$.

\noindent
The following theorem is proved by a method similar to
the proof of Theorem \ref{galrepirred} (\cf [Mo4]).

\begin{thm}[{\cite[Theorem 1.4]{A3}}]
\label{M0BpK}

Let $K$ be a finite Galois extension of $\Q$ which does not contain
the Hilbert class field of any imaginary quadratic field.
Assume that there is a prime number $q$ which splits completely in $K$ and satisfies
$B\otimes_{\Q}\Q(\sqrt{-q})\not\cong\M_2(\Q(\sqrt{-q}))$.
Let $p>4q$ be a prime number which also satisfies
$p\geq 11$, $p\ne 13$, $p\nmid d$ and $p\not\in \cN_1(K)\cup\cN'_1(K)$.

(1)
If $B\otimes_{\Q}K\cong\M_2(K)$, then $M_0^B(p)(K)=\emptyset$.

(2)
If $B\otimes_{\Q}K\not\cong\M_2(K)$, then
$M_0^B(p)(K)\subseteq\{\text{elliptic points of order $2$ or $3$}\}$.

\end{thm}

Here an elliptic point of order $2$ (\resp $3$) is a point
whose corresponding triple $(A,i,V)$ (over $\Kb$) satisfies
$\Aut_{\cO}(A,V)\cong\Z/4\Z$ (\resp $\Z/6\Z$),
where $\Aut_{\cO}(A,V)$ is the group of automorphisms of $A$
defined over $\Kb$ commuting with the action of $\cO$
and stabilizing $V$.

\section{Examples}
\label{secEx}

We give several explicit examples of Theorem \ref{main} in the following proposition.

\begin{prp}
\label{exRT}

Let $d\in\{6,10,22\}$ and $K\in\{\Q(\sqrt{3},\sqrt{-5}), \Q(\zeta_5), \Q(\zeta_{17})\}$.
Assume $(d,K)\ne(22,\Q(\zeta_5))$.
Then there are infinitely many $\Kb$-isomorphism classes of QM-abelian surfaces
$(A,i)$ by $\cO$ over $K$, and the set $\sA(K,2)_B$ is finite.

\end{prp}

To prove Proposition \ref{exRT}, we need the following four lemmas.

\begin{lmm}
\label{exqsplit}

Let $K$ be $\Q(\sqrt{3},\sqrt{-5})$ (\resp $\Q(\zeta_5)$, \resp $\Q(\zeta_{17})$).
Then a prime number $q$ splits completely in $K$ if and only if
$q\equiv 1,23,47,49\bmod{60}$
(\resp $q\equiv 1\bmod{5}$, \resp $q\equiv 1\bmod{17}$).

\end{lmm}

\begin{proof}

A prime number $q$ splits in $\Q(\sqrt{3})$ (\resp $\Q(\sqrt{-5})$)
if and only if
$q\equiv\pm 1\bmod{12}$
(\resp $q\equiv 1,3,7,9\bmod{20}$).
Then the assertion for $\Q(\sqrt{3},\sqrt{-5})$ follows.
The rest of the assertions are trivial.

\end{proof}

\begin{lmm}
\label{exB-qM2}

Let $d$ be $6$ (\resp $10$, \resp $22$).
For a prime number $q$, we have
$B\otimes_{\Q}\Q(\sqrt{-q})\not\cong\M_2(\Q(\sqrt{-q}))$
if and only if
$q\equiv 2,5,7,11,17,23\bmod{24}$

\noindent
(\resp $q\equiv 1,7,9,11,19,21,23,29,31,39\bmod{40}$, 

\noindent
\resp $q\equiv 2,7,13,15,17,19,21,23,29,31,35,39,41,43,
47,51,57,61,63,65,71,73,79,83,85,87$

\noindent
$\bmod{88}$).

\end{lmm}

\begin{proof}

For a quadratic field $L$, we have $B\otimes_{\Q}L\not\cong\M_2(L)$
if and only if there is a prime divisor of $d$ which splits in $L$.
The prime number $2$ (\resp $3$, \resp $5$, \resp $11$)
splits in $\Q(\sqrt{-q})$ if and only if
$q\equiv -1\bmod{8}$
(\resp $q\equiv -1\bmod{3}$,
\resp $q\equiv \pm 1\bmod{5}$,
\resp $q\equiv 2,6,7,8,10\bmod{11}$).
Then we have done.

\end{proof}

\begin{lmm}
\label{exMBKinfty}

Let $d\in\{6,10,22\}$ and $K\in\{\Q(\sqrt{3},\sqrt{-5}), \Q(\zeta_5), \Q(\zeta_{17})\}$.
Assume $(d,K)\ne(22,\Q(\zeta_5))$.
Then $\sharp M^B(K)=\infty$.

\end{lmm}

\begin{proof}

It suffices to show $M^B(K)\ne\emptyset$.
Looking at (\ref{eqMB}), it is enough to show
$M^B(K_v)\ne\emptyset$ for any place $v$ of $K$,
owing to the Hasse principle.
If $v$ is infinite, it is trivial since $K_v=\C$.
For $d=6$ (\resp $d=10$, \resp $d=22$) and a prime number $p$,
we have $M^B(\Q_p)=\emptyset$ if and only if
$p=3$ (\resp $p=2$, \resp $p=11$).
(To show $M^B(\Q_p)\ne\emptyset$, if $p\ne 2$, consider
the equations in (\ref{eqMB}) modulo $p$ and use Hensel's lemma;
if $p=2$, find explicit solutions of the equations
$(\sqrt{-7})^2+2^2+3=0$ with $\sqrt{-7}\in\Q_2$
and
$(\sqrt{-15})^2+2^2+11=0$ with $\sqrt{-15}\in\Q_2$.
To show $M^B(\Q_p)=\emptyset$, we use the fact that
the equation $x^2+y^2+p=0$ has a solution in $\Q_p$
if and only if $p\equiv 1\bmod{4}$.)
For any quadratic extension $L$ of $\Q_p$,
we have $M^B(L)\ne\emptyset$.
So, for $d=6$ (\resp $d=10$, \resp $d=22$), it suffices to show
that $K_v$ contains a quadratic extension of $\Q_3$ (\resp $\Q_2$, \resp $\Q_{11}$)
for any place $v$ of $K$ above $3$ (\resp $2$, \resp $11$).

For a prime number $p$, let $e_p(K)$ (\resp $f_p(K)$, \resp $g_p(K)$)
be the ramification index of $p$ in $K/\Q$
(\resp the degree of the residual field extension above $p$ in $K/\Q$,
\resp the number of primes of $K$ above $p$).
For

$K=\Q(\sqrt{3},\sqrt{-5})$ (\resp $\Q(\zeta_5)$, \resp $\Q(\zeta_{17})$),
we have

$(e_3(K),f_3(K),g_3(K))=(2,1,2)$ (\resp $(1,4,1)$, \resp $(1,16,1)$),

$(e_2(K),f_2(K),g_2(K))=(2,1,2)$ (\resp $(1,4,1)$, \resp $(1,8,2)$),

$(e_{11}(K),f_{11}(K),g_{11}(K))=(1,2,2)$ (\resp $(\textbf{1},\textbf{1},\textbf{4})$, \resp $(1,16,1)$).

\noindent
Then $K_v$ contains a quadratic extension of $\Q_3$ (\resp $\Q_2$, \resp $\Q_{11}$)
for any place $v$ of $K$ above $3$ (\resp $2$, \resp $11$)
unless $K=\Q(\zeta_5)$ and $v|11$.
Note that if $K=\Q(\zeta_5)$ and $v|11$, then $K_v=\Q_{11}$.
For the proof of the next lemma, we add

$(e_5(K),f_5(K),g_5(K))=(2,2,1)$ (\resp $(4,1,1)$, \resp $(1,16,1)$).

\end{proof}

\begin{lmm}
\label{exBKM2}

Let $d\in\{6,10,22\}$ and $K\in\{\Q(\sqrt{3},\sqrt{-5}), \Q(\zeta_5), \Q(\zeta_{17})\}$.
Assume $(d,K)\ne(22,\Q(\zeta_5))$.
Then $B\otimes_{\Q}K\cong\M_2(K)$.

\end{lmm}

\begin{proof}

It suffices to show $B\otimes_{\Q}K_v\cong\M_2(K_v)$
for any place $v$ of $K$.
It is trivial if $v$ is infinite, or if $v$ is finite and does not divide $d$.
By the computation in the proof of Lemma \ref{exMBKinfty},
no prime divisor of $d$ splits completely in $K$
unless $(d,K)=(22,\Q(\zeta_5))$.
So, if $(d,K)\ne(22,\Q(\zeta_5))$, and if $v$ is finite and divides $d$,
then $K_v$ contains a quadratic extension of
$\Q_{p(v)}$,
where $p(v)$ is the residual characteristic of $v$.
In such a case, $B\otimes_{\Q}K_v\cong\M_2(K_v)$.

\end{proof}

\noindent
(Proof of Proposition \ref{exRT})

The only imaginary quadratic subfields of $\Q(\sqrt{3},\sqrt{-5})$
are $\Q(\sqrt{-5})$ and $\Q(\sqrt{-15})$, which are not of class number one.
Since the extension $\Q(\sqrt{3},\sqrt{-5})/\Q(\sqrt{-5})$
(\resp $\Q(\sqrt{3},\sqrt{-5})/\Q(\sqrt{-15})$)
is ramified over the primes above $3$ (\resp $2$),
the field $\Q(\sqrt{3},\sqrt{-5})$ is not the Hilbert class field of
$\Q(\sqrt{-5})$ (\resp $\Q(\sqrt{-15})$).
The only quadratic subfield of $\Q(\zeta_5)$ (\resp $\Q(\zeta_{17})$)
is $\Q(\sqrt{5})$ (\resp $\Q(\sqrt{17})$).
So, none of $\Q(\sqrt{3},\sqrt{-5}), \Q(\zeta_5), \Q(\zeta_{17})$
contains the Hilbert class field of any imaginary quadratic field.
By Lemmas \ref{exqsplit} and \ref{exB-qM2}, there is a prime number $q$
which splits completely in $K$ and satisfies 
$B\otimes_{\Q}\Q(\sqrt{-q})\not\cong\M_2(\Q(\sqrt{-q}))$.
Then Lemma \ref{irredRT} and Theorem \ref{galrepirred} imply
$\sharp\sA(K,2)_B<\infty$.
By the remark at the end of \S \ref{secMB}, together with Lemmas \ref{exMBKinfty} and \ref{exBKM2},
there are infinitely many $\Kb$-isomorphism classes of QM-abelian surfaces
$(A,i)$ by $\cO$ over $K$.



\begin{thebibliography}{99}
%


\bibitem{A1} Arai,~K.,
On the Galois images associated to QM-abelian surfaces,
\textit{Proceedings of the Symposium on Algebraic Number Theory and Related Topics}, 165--187,
\textit{RIMS K\^{o}ky\^{u}roku Bessatsu}, \textbf{B4}, Res. Inst. Math. Sci. (RIMS), Kyoto, 2007.


\bibitem{A2} Arai,~K.,
Galois images and modular curves,
\textit{Proceedings of the Symposium on Algebraic Number Theory and Related Topics}, 145--161,
\textit{RIMS K\^{o}ky\^{u}roku Bessatsu}, \textbf{B32}, Res. Inst. Math. Sci. (RIMS), Kyoto, 2012.

\bibitem{A3} Arai,~K.,
Algebraic points on Shimura curves of $\Gamma_0(p)$-type II,
\textit{preprint},
available at the web page (http://arxiv.org/pdf/1205.3596v2.pdf).

\bibitem{AM} Arai,~K. and Momose,~F.,
Algebraic points on Shimura curves of $\Gamma_0(p)$-type,
\textit{J. Reine Angew. Math.},
published online (ahead of print).

\bibitem{AM2} Arai,~K. and Momose,~F.,
Errata to Algebraic points on Shimura curves of $\Gamma_0(p)$-type,
\textit{J. Reine Angew. Math.},
published online (ahead of print).

\bibitem{Bu} Buzzard,~K.,
Integral models of certain Shimura curves, 
\textit{Duke Math. J.} \textbf{87} (1997), no. 3, 591--612.





\bibitem{Fa} Faltings,~G.,
Endlichkeitss\"{a}tze f\"{u}r abelsche Variet\"{a}ten \"{u}ber Zahlk\"{o}rpern,
\textit{Invent. Math.} \textbf{73} (1983), no. 3, 349--366.


\bibitem{J} Jordan,~B.,
Points on Shimura curves rational over number fields,
\textit{J. Reine Angew. Math.} \textbf{371} (1986), 92--114.





\bibitem{Kur} Kurihara,~A,
On some examples of equations defining Shimura curves
and the Mumford uniformization,
\textit{J. Fac. Sci. Univ. Tokyo Sect. IA Math.} \textbf{25} (1979), no. 3, 277--300. 




\bibitem{Ma3} Mazur,~B.,
Rational isogenies of prime degree (with an appendix by D. Goldfeld), 
\textit{Invent. Math.} \textbf{44} (1978), no. 2, 129--162.



\bibitem{Mo4} Momose,~F.,
Isogenies of prime degree over number fields,
\textit{Compositio Math.} \textbf{97} (1995), no. 3, 329--348.




\bibitem{Oh} Ohta,~M., 
On $l$-adic representations of Galois groups obtained from certain 
two-dimensional abelian varieties, 
\textit{J. Fac. Sci. Univ. Tokyo Sect. IA Math.} \textbf{21} (1974), 299--308.

\bibitem{Oz1} Ozeki,~Y.,
Non-existence of certain Galois representations with a uniform tame inertia weight,
\textit{Int. Math. Res. Not.} \textbf{2011} (2011), no. 11, 2377--2395.

\bibitem{Oz2} Ozeki,~Y.,
Non-existence of certain CM abelian varieties with prime
power torsion,
\textit{preprint},
available at the web page (http://arxiv.org/pdf/1112.3097.pdf).



\bibitem{RaTa} Rasmussen,~C. and Tamagawa,~A.,
A finiteness conjecture on abelian varieties with constrained prime power torsion,
\textit{Math. Res. Lett.} \textbf{15} (2008), no. 6, 1223--1231. 







\bibitem{SeTa} Serre,~J.-P. and Tate,~J.,
Good reduction of abelian varieties,
\textit{Ann. of Math. (2)} \textbf{88} (1968), 492--517. 

\bibitem{Sh} Shimura,~G.,
On the real points of an arithmetic quotient of a bounded symmetric domain,
\textit{Math. Ann.} \textbf{215} (1975), 135--164.


\bibitem{Z} Zarhin,~Y.\,G.
A finiteness theorem for unpolarized abelian varieties over number fields
with prescribed places of bad reduction,
\textit{Invent. Math.} \textbf{79} (1985), no. 2, 309--321.


\end{thebibliography}
\end{document}